\documentclass[a4paper,twoside]{amsart}

\usepackage{amssymb}
\usepackage[mathscr]{eucal}
\usepackage{hyperref}

\theoremstyle{plain}
\newtheorem{prop}{Proposition}[section]
\newtheorem{thm}[prop]{Theorem}
\newtheorem{cor}[prop]{Corollary}
\newtheorem{conj}[prop]{Conjecture}

\theoremstyle{definition}

\newtheorem{lab}[prop]{}

\theoremstyle{remark}
\newtheorem{rem}[prop]{Remark}
\newtheorem{rems}[prop]{Remarks}

\renewcommand{\subset}{\subseteq}

\newcommand{\C}{{\mathbb{C}}}
\renewcommand{\P}{{\mathbb{P}}}
\newcommand{\R}{{\mathbb{R}}}

\newcommand{\x}{{\mathtt{x}}}
\newcommand{\sfS}{\mathsf{S}}

\DeclareMathOperator{\ch}{char}
\DeclareMathOperator{\Hom}{Hom}
\DeclareMathOperator{\im}{im}
\DeclareMathOperator{\interior}{int}

\newcommand{\du}{{\scriptscriptstyle\vee}}

\newcommand{\ol}{\overline}

\renewcommand{\setminus}{\smallsetminus}
\renewcommand{\epsilon}{\varepsilon}
\renewcommand{\theta}{\vartheta}

\newcommand{\ex}{\exists\,}

\renewcommand{\choose}[2]{\genfrac(){0pt}{}{#1}{#2}}

\begin{document}

\title
[Sum of squares length of real forms]
{Sum of squares length of real forms}

\author
{Claus Scheiderer}

\address
 {Fachbereich Mathematik und Statistik \\
 Universit\"at Konstanz \\
 78457 Konstanz \\
 Germany}
\email
 {claus.scheiderer@uni-konstanz.de}

\begin{abstract}
For $n,\,d\ge1$ let $p(n,2d)$ denote the smallest number $p$ such
that every sum of squares of forms of degree $d$ in
$\R[x_1,\dots,x_n]$ is a sum of $p$ squares. We establish lower
bounds for these numbers that are considerably stronger than the
bounds known so far. Combined with known upper bounds they give
$p(3,2d)\in\{d+1,\,d+2\}$ in the ternary case. Assuming a conjecture
of Iarrobino-Kanev on dimensions of tangent spaces to catalecticant
varieties, we show that $p(n,2d)\sim const\cdot d^{(n-1)/2}$ for
$d\to\infty$ and all $n\ge3$.
For ternary sextics and quaternary quartics we determine the exact
value of the invariant, showing $p(3,6)=4$ and $p(4,4)=5$.
\end{abstract}

\maketitle


\section*{Introduction}

Given a polynomial $f(x_1,\dots,x_n)$ with real coefficients that
has nonnegative values, Artin proved that $f$ can be written as a
sum of squares of rational functions over $\R$. By a famous theorem
of Pfister \cite{pf}, $2^n$ squares are always sufficient to
represent $f$. For $n\le2$ this bound is known to be best possible,
whereas for $n\ge3$ it is only known that the best general bound
lies between $n+2$ and $2^n$ (see \cite{s}, for instance).

In general it is not possible to represent $f$ as a sum of squares of
polynomials, as Hilbert \cite{hi} proved in 1888. However if we
restrict attention to sums of squares of polynomials, it is natural
to ask for an upper bound on the number of squares needed.
Switching to homogeneous polynomials (forms), we therefore fix
$n,\,d\ge1$ and ask for the smallest number $p=p(n,2d)$ such that
every form of degree $2d$ in $\R[x_1,\dots,x_n]$ that is a sum of
squares of forms is a sum of $p$ such squares. For $n\le2$ or $d=1$
this is an elementary question, but otherwise the only case where
$p(n,2d)$ is known exactly is $(n,2d)=(3,4)$, where $p(3,4)=3$
according to Hilbert \cite{hi}.

Still quite a bit is known for other combinations $(n,2d)$. Choi, Lam
and Reznick \cite{clr} established both upper and lower bounds for
$p(n,2d)$. Fixing the number $n$ of variables, they proved in
particular that $p(n,2d)=O(d^{(n-1)/2})$ for $d\to\infty$.
While the upper bounds from \cite{clr} are still largely the best
ones known when $n\ge4$, there exists a substantial improvement for
$n=3$, due to Leep \cite{lp}. Using Cassels-Pfister theory over a
univariate polynomial ring, he improved the upper bound for $p(3,2d)$
from $2d+1$ to $d+2$. On the other hand, there is a huge gap between
known upper and lower bounds. Fixing $n$ and letting $d$ grow, the
asymptotically best lower bounds known for $p(n,2d)$ are only
logarithmic in $d$~\cite{cdlr}. See Section~\ref{secreview} below for
precise details.

The first main result of this paper establishes new lower bounds that
are much closer to the existing upper bounds. In fact they have the
same growth rate for $d\to\infty$, showing $p(n,2d)\sim d^{(n-1)/2}$
for $d\to\infty$. For ternary forms, they combine with Leep's theorem
to show that $p(3,2d)$ is either $d+1$ or $d+2$. The idea is to
consider forms $f$ that are sums of squares and vanish in a generic
finite set $Z\subset\P^{n-1}(\R)$ of appropriate size. For such $f$
we show that the sum of squares representation of $f$ is essentially
unique. The explicit form of the resulting bound depends
on the assumption that the Hilbert series of the squared vanishing
ideal $I(Z)^2$ is the expected one. For $n=3$ this has been proved,
but for $n\ge4$ it is merely a conjecture due to Iarrobino-Kanev.
Although the conjecture has been verified for small values of $n$ and
$d$, our bounds for $n\ge4$ depend therefore in general on this
conjecture.

The second main result computes the invariants $p(3,6)$ and $p(4,4)$
exactly, showing $p(3,6)=4$ and $p(4,4)=5$. Note that $(3,6)$ and
$(4,4)$ are precisely the two minimal cases where the sums of squares
cone is strictly smaller than the cone of nonnegative forms.

The paper is organized as follows. Section~\ref{secreview} recalls
the upper and lower bounds for $p(n,2d)$ that can be found in the
literature. Section~\ref{sectyp} studies the collection of typical
sums of squares lengths of forms in $\Sigma_{n,2d}$, i.e., lengths
that occur for an open nonempty set of forms. We show that every
integer between the typical complex length (essentially known) and
the maximum length $p(n,2d)$ is a typical length in this sense.
Section~\ref{sectlowerbd} establishes the new lower bound for
$p(n,2d)$ (depending on the Iarrobino-Kanev conjecture for $n\ge4$)
and discusses the asymptotics of $p(n,2d)$. In Section~\ref{sec3644}
we show that in the cases $(3,6)$ and $(4,4)$ the existing lower
bound for $p(n,2d)$ is sharp.

We remark that the results of this paper remain true over any real
closed field, instead of $\R$.


\section{Review of known upper and lower bounds}\label{secreview}%

\begin{lab}
We start with setting up basic notation.
Let $n\ge1$, $e\ge0$, and write $\x=(x_1,\dots,x_n)$. By $\R[\x]_e$
we denote the space of degree~$e$ forms (homogeneous polynomials) in
$\R[\x]$. We often abbreviate
$$N_{n,e}\>:=\>\dim\R[\x]_e\>=\>\choose{n+e-1}{n-1}.$$
If $e=2d$ is even then $\Sigma_{n,2d}$ denotes the subset
of $\R[\x]_{2d}$ of all forms that are sums of squares of forms of
degree~$d$. It is well known that $\Sigma_{n,2d}$ is a closed convex
cone in $\R[\x]_{2d}$ with nonempty interior.
Given $f\in\Sigma_{n,2d}$, we denote by
$$\ell(f)\>=\>\min\bigl\{r\ge0\colon\ex p_1,\dots,p_r\in\R[\x]_d
\text{ with }f=p_1^2+\cdots+p_r^2\bigr\}$$
the \emph{sum of squares length} (or \emph{sos length}) of $f$. We
call
$$p(n,2d)\>:=\>\sup\bigl\{\ell(f)\colon f\in\Sigma_{n,2d}\bigr\}$$
the \emph{Pythagoras number} of $n$-ary forms of degree $2d$.
\end{lab}

\begin{lab}
It is elementary to see $p(1,2d)=1$, $p(2,2d)=2$ ($d\ge1$) and
$p(n,2)=n$ ($n\ge1$). Clearly, $n\le m$ and $d\le e$ imply $p(n,2d)
\le p(m,2e)$.
Hilbert \cite{hi} proved $p(3,4)=3$. So far, these are the only cases
where the precise value of $p(n,2d)$ is known.
\end{lab}

\begin{lab}\label{cdlr}%
For $n\ge2$, Choi, Dai, Lam and Reznick \cite{cdlr} proved in 1982
that the ring $\R[x_1,\dots,x_n]$ has Pythagoras number $\infty$.
This amounts to $p(n,2d)\to\infty$ for $n\ge3$ and $d\to\infty$. From
the proof of \cite{cdlr} Theorem 4.10 one can extract explicit lower
bounds for $p(n,2d)$ when $n\ge3$. In particular, it is shown there
that $p(n,4d+2)>p(n,2d)$, resulting in lower bounds for $p(n,2d)$
that are logarithmic in $d$.
\end{lab}

A systematic study of the invariants $p(n,2d)$ was initiated by Choi,
Lam and Reznick \cite{clr} in 1995.
Adopting notation from \cite{clr}, the following is essentially the
main result:

\begin{thm}\label{clrbounds}%
Let $n,\,d\ge1$. The Pythagoras number $p:=p(n,2d)$ satisfies the
inequalities
\begin{equation}\label{genbds}%
\choose{p+1}2\>\le\>N_{n,2d}\>\le\>pN_{n,d}-\choose p2.
\end{equation}
These translate into
\begin{equation}\label{genbds2}%
\lambda(n,2d)\>\le\>p(n,2d)\>\le\>\Lambda(n,2d)
\end{equation}
where (writing $a=N_{n,2d}$ and $e=N_{n,d}$)
\begin{equation}\label{lambdadefs}%
\lambda(n,2d)\>:=\>\frac12\Bigl(2e+1-\sqrt{(2e+1)^2-8a}\Bigr),
\quad\Lambda(n,2d)\>:=\>\frac12\Bigl(-1+\sqrt{1+8a}\Bigr).
\end{equation}
\end{thm}

\begin{proof}
(Sketch)
The first inequality in \eqref{genbds} comes from the fact that when
$f=\sum_{i=1}^rf_i^2$ is a sum of squares representation of $f$ of
minimal length, the products $f_if_j$ are linearly independent.
The second follows from the fact that the sum of squares map
$(\R[\x]_d)^p\to\R[\x]_{2d}$ is submersive generically on the source,
according to Sard's theorem.
\end{proof}

\begin{lab}
The setup of \cite{clr} is in fact more general than in Theorem
\ref{clrbounds}. The authors consider sums of squares $f$ whose
Newton polytope $\text{New}(f)$ is contained in a fixed convex set
(``cage'') $C$. Inequalities \eqref{genbds2} are generalized (with
essentially the same proof) to
$$\lambda(C)\>\le\>p(C)\>\le\>\Lambda(C)$$
for any cage $C$, where $p(C)=\max\{\ell(f)\colon f\in\Sigma$,
$\text{New}(f)\subset C\}$, and $\lambda(C)$, $\Lambda(C)$ are
defined as in \eqref{lambdadefs}, with $e$ resp.\ $a$ (related to)
the numbers of lattice points in $\frac12C$ and $C$.
See \cite{clr} for precise definitions.
\end{lab}

\begin{lab}\label{remlowbds}%
Particular choices of $C$ may lead to improved lower bounds for
$p(n,2d)$. Namely, let $\ol\lambda(n,2d)=\max_C\lambda(C)$, maximum
over the Newton polytopes $C$ of forms in $\R[x_1,\dots,x_n]_{2d}$.
Then $\ol\lambda(n,2d)\le p(n,2d)$, and it may happen that
$\lceil\lambda(n,2d)\rceil<\lceil\ol\lambda(n,2d)\rceil$. This
phenomenon was already discussed in \cite{clr},
and was later studied in detail by Leep and Starr \cite{ls}.

Unfortunately there is a huge discrepancy between lower and upper
bounds. The reason is that the lower bounds $\lambda$ are quite weak
in general, even in the improved version $\ol\lambda$.
Indeed, it is easy to see that $d\mapsto\lambda(n,2d)$ is growing
with limit $2^{n-1}$ for $d\to\infty$,
and it is expected (\cite{ls} Conjecture 4.3) that the
same holds for $\ol\lambda$. For $n=3$ is has been shown that
$\ol\lambda(n,2d)=\lceil\lambda(n,2d)\rceil=4$ for all $d\ge3$
(\cite{ls} Theorem 4.3).
Therefore, when $d$ is large relative to
$n$, then $\lambda$ (and probably $\ol\lambda$ as well) only gives
the lower bound $2^{n-1}$ for $p(n,2d)$, which does not even depend
on $d$. This strongly contrasts with the fact that $p(n,2d)\to\infty$
for $d\to\infty$, when $n\ge3$. See also the discussion at the end of
\cite{ls}. Altogether, when one considers the general case ($d$ large
relative to $n\ge3$), it seems that the strongest known lower bounds
for $p(n,2d)$ are still the logarithmic bounds derived from
\cite{cdlr}, see \ref{cdlr} above.
\end{lab}

\begin{lab}\label{gnrappl}%
We now discuss a substantial improvement for the upper bound in the
case of ternary forms, due to Leep \cite{lp}. To describe his result,
let $A$ be a ring containing $\frac12$. For $r\ge1$ let $g_r(A)$
denote the smallest number $g\ge1$ such that, for any finite number
$l_1,\dots,l_N$ of linear forms in $A[x_1,\dots,x_r]$, there exist
$g$ other linear forms $l'_1,\dots,l'_g\in A[x_1,\dots,x_r]$ with
$l_1^2+\cdots+l_N^2=l_1'^2+\cdots+l_g'^2$.
If there is no finite such bound we write $g_r(A)=\infty$.

For example, when $R$ is a real closed field and $t$ is a variable,
$g_r(R[t])=r+1$ for every $r\ge1$ (\cite{blop} Example~iii).
\end{lab}

\begin{lab}\label{remgr}%
Upper bounds on $g_r$ can lead to upper bounds for sos lengths, by
the following elementary observation (compare \cite{cdlr} Theorem
2.7).
Let $B$ be an $A$-algebra, and let $M$ be an $A$-submodule of $B$
generated by $r<\infty$ elements. Let $g=g_r(A)$. If $b\in B$ is a
sum of squares of elements of $M$, then there exist $b_1,\dots,b_g\in
M$ with $b=b_1^2+\cdots+b_g^2$.
\end{lab}

Leep proves (\cite{lp} Theorem 5.2):

\begin{thm}\label{leepthm}%
(Leep)
If $k$ is any real field then $g_r(k[t])=g_r(k(t))$ for all $r\ge1$.
\end{thm}

\begin{cor}\label{leepcor}%
When $R$ is a real closed field then $g_r(R[t])=r+1$ for any $r\ge1$.
\qed
\end{cor}

Corollary \ref{leepcor} remains true when the field $R$ is merely
hereditarily pythagorean, see \cite{lp} Theorem 6.3 and thereafter.

\begin{lab}
When we speak of zeros of a form $f(x_1,\dots,x_n)$ we mean zeros in
complex projective space $\P^{n-1}(\C)$. Recall that the multiplicity
of a form $0\ne f\in\R[x_1,\dots,x_n]$ at a point $\xi\in\P^{n-1}
(\C)$ is the minimal number $m\ge0$ for which there exists an
$n$-tuple $a=(a_1,\dots,a_n)$ of nonnegative integers with
$\sum_{i=1}^na_i=m$ and with $(\partial^af)(\xi)=\partial_1^{a_1}
\cdots\partial_n^{a_n}f(\xi)\ne0$. When $f$ has nonnegative values,
the multiplicity of $f$ at any point in $\P^{n-1}(\R)$ is even. The
form $f$ is singular if and only if it has multiplicity $\ge2$ at
some $\xi\in\P^{n-1}(\C)$.
\end{lab}

Applying Corollary \ref{leepcor} we get:

\begin{thm}\label{leepboundallg}%
Let $n\ge2$ and $d\ge1$, let $f\in\Sigma_{n,2d}$, and assume that $f$
has a real zero of multiplicity $2m\ge0$. Then
$$\ell(f)\>\le\>1+\choose{n+d-2}{n-2}-\choose{n+m-3}{n-2}.$$
In particular, $m=0$ gives
\begin{equation}\label{leepbdallg}
p(n,2d)\>\le\>1+\choose{n-2+d}{n-2}
\end{equation}
for all $d\ge1$.
\end{thm}

The most interesting case is when $n=3$ (ternary forms):

\begin{cor}\label{ternarybound}%
(Leep)
If $f\in\Sigma_{3,2d}$ has a real zero of multiplicity $2m\ge0$ then
$\ell(f)\le d+2-m$. In particular, $p(3,2d)\le d+2$ for all $d\ge1$.
\qed
\end{cor}

\begin{proof}[Proof of Theorem \ref{leepboundallg}]
Let $\xi\in\P^{n-1}(\R)$ be a real zero of $f$ of multiplicity
$2m\ge0$. After a linear change of coordinates we can assume $\xi=
(0:\cdots:0:1)$. Then $\deg_{x_n}(f)=2(d-m)$.
If $f$ is written as a sum of squares, say $f=\sum_\nu p_\nu^2$,
the $p_\nu$ contain only monomials $\x^\beta$ with $|\beta|=d$ and
$0\le\beta_n\le d-m$. Let $\R[x_1,\dots,x_n]\to\R[x_2,\dots,x_n]$,
$g\mapsto g':=g(1,x_2,\dots,x_n)$ be the dehomogenization operator.
The $p'_\nu$ lie in the
$\R[x_2]$-submodule of $\R[x_2,\dots,x_n]$ spanned by the monomials
$x_3^{\beta_3}\cdots x_n^{\beta_n}$
with $\beta_3+\cdots+\beta_n\le d$ and $\beta_n\le d-m$. The number
of these monomials is
$$\sum_{j=0}^{d-m}\choose{n-3+d-j}{n-3}\>=\>\choose{n-2+d}{n-2}
-\choose{n-3+m}{n-2}\>=:\>N.$$
Using Remark \ref{remgr} and Corollary \ref{leepcor} we conclude
$\ell(f)\le N+1$.
\end{proof}

\begin{rems}\label{uppbdvgl}%
\hfil
\smallskip

1.\
Denote the bound from Theorem \ref{leepboundallg} by $L(n,2d):=1+
\choose{n+d-2}{n-2}$. Let us compare $L(n,2d)$ with the general
upper bound $\Lambda(n,2d)$ (Theorem \ref{clrbounds}). When $n=3$
we have $L(3,2d)=d+2$, which is significantly better than
$\Lambda(3,2d)=2d+1$. For $n\ge4$, however $L(n,2d)$ is usually
weaker (bigger) than $\Lambda(n,2d)$. This is clear since $L(n,2d)$
is a polynomial of degree $n-2$ in $d$, whereas $\Lambda(n,2d)=
O(d^{(n-1)/2})$.
In fact, there are only three pairs $(n,d)$ with $n\ge4$ and $d\ge1$
where $L(n,2d)<\lfloor\Lambda(n,2d)\rfloor$. These are
$$(n,2d)\>=\>(4,6),\ (4,8)\text{ and }(4,10),$$
where $\lfloor\Lambda(n,2d)\rfloor=12$, $17$, $23$ respectively,
while $L(n,2d)$ is smaller by one.
\smallskip

2.\
As far as we know, the upper bounds $p(3,2d)\le d+2$ for $d\ge2$
(\ref{ternarybound}), resp.\
$$p(n,2d)\>\le\>\begin{cases}L(n,2d)&n=4\text{ and }3\le d\le5,\\
\Lambda(n,2d)&\text{otherwise}\end{cases}$$
for $n\ge4$ and $d\ge2$, are the best ones known to date.
\end{rems}


\section{Typical sum of squares lengths}\label{sectyp}%

Let $\x=(x_1,\dots,x_n)$. We show that the set of typical sos lengths
for forms in $\R[\x]_d$ is determined by $p(n,2d)$ and by the
(unique) typical sos length over $\C$.

\begin{lab}
Fix $n,\,d\ge1$, and work over $\C$ first. Let $t=t(n,2d)\ge1$ be the
minimal number for which the set $\{p_1^2+\cdots+p_t^2\colon p_1,
\dots,p_t\in\C[\x]_d\}$ is (Zariski) dense in $\C[\x]_{2d}$. By
Sard's theorem, this is also the minimal number $t$ of forms
$p_1,\dots,p_t$ of degree $d$ with $\langle x_1,\dots,x_n\rangle^{2d}
\subset\langle p_1,\dots,p_r\rangle$,
where $\langle\cdots\rangle$ denotes the ideal generated in $\C[\x]$.
\end{lab}

\begin{lab}
It is clear that $t(2,2d)=2$ for all $d\ge1$ and $t(n,2)=n$ for all
$n\ge1$, and that $t=t(n,2d)$ always satisfies $tN_{n,d}-\choose t2
\ge N_{n,2d}$, i.e.\ $t\ge\lambda(n,2d)$ (see \ref{clrbounds}).
A~general conjecture due to Fr\"oberg \cite{fb} predicts the Hilbert
series of an ideal generated by a generic collection of forms of
prescribed degrees. As a very particular case, this conjecture
predicts $t(n,2d)=\lceil\lambda(n,2d)\rceil$ for all $(n,d$).
For all $(n,d)$, the theorem of Fr\"oberg, Ottaviani and Shapiro
\cite{fos} implies $t(n,2d)\le2^{n-1}$.
In particular we have $t(n,2d)=2^{n-1}$ when $\lceil\lambda(n,2d)
\rceil=2^{n-1}$, which is the case for $d$ large enough. For smaller
$d$ the values of $t(n,2d)$ may be determined with the help of a
computer algebra system, as long as $n$ is not too big. In this way
one shows that $t(3,2)=t(3,4)=3$ and $t(3,2d)=4$ for $2d\ge6$. For
$n=4$ the values of $t$ are
$$t(4,2d)\ =\ \begin{cases}5&\mbox{if \ }2\le d\le4,\\6&\mbox{if \ }
5\le d\le8,\\7&\mbox{if \ }9\le d\le20,\\8&\mbox{if \ }d\ge21.
\end{cases}$$
\end{lab}

\begin{lab}
Now fix $n,\,d\ge1$. For $r\ge1$ let
$$\Sigma_{n,2d}(r)\>:=\>\{f\in\Sigma_{n,2d}\colon\ell(f)\le r\},$$
and let
$$T(n,2d)\>:=\>\Bigl\{r\ge1\colon\Sigma_{n,2d}(r)\setminus
\Sigma_{n,2d}(r-1)\text{ has nonempty interior in }\R[\x]_{2d}
\Bigr\}.$$
$T(n,2d)$ is the set of \emph{typical sos lengths} for
$\Sigma_{n,2d}$, i.e., the set of numbers $r$ for which there exists
a non-empty open set consisting of forms of sos length $r$. Clearly
$t(n,2d)$ is the smallest element of $T(n,2d)$.
\end{lab}

\begin{prop}\label{equidim}%
Let $n,\,d,\,r\ge1$, and let $m_r=\dim(I_{2d})$ where $I$ is an ideal
generated by a generic sequence of $r$ forms of degree~$d$.
\begin{itemize}
\item[(a)]
$\Sigma_{n,2d}(r)$ is a closed semi-algebraic subset of
$\R[\x]_{2d}$.
\item[(b)]
For every $f\in\Sigma_{n,2d}(r)$, the local dimension of $\Sigma
_{n,2d}(r)$ at $f$ is $m_r$.
\item[(c)]
If $r\ge t(n,2d)$ then $\Sigma_{n,2d}(r)$ is the closure of its
interior.
\end{itemize}
\end{prop}

\begin{proof}
Consider the map $\phi\colon(\R[\x]_d)^r\to\R[\x]_{2d}$,
$(p_1,\dots,p_r)\mapsto\sum_{i=1}^rp_i^2$. (a)~is clear since $\phi$
is a proper polynomial map. In (b) it is clear that
$\dim\Sigma_{n,2d}(r)\le m_r$. For all $r$-tuples $p=(p_1,\dots,p_r)
\in(\R[\x]_d)^r$ outside a proper real algebraic set we have
$\dim\langle p_1,\dots,p_r\rangle_{2d}=m_r$. For these $p$, the local
dimension of $\Sigma_{n,2d}(r)$ at $\phi(p)$ is equal to $m_r$. Since
these $p$ are dense in $\im(\phi)$ we get~(b). For $r\ge t(n,2d)$ we
have $m_r=N_{n,2d}$, so (b) implies~(c).
\end{proof}

\begin{cor}\label{typle}%
The typical sos lengths for $\Sigma_{n,2d}$ are $T(n,2d)=\{t,\,t+1,\,
\dots,\,p\}$ where $t=t(n,2d)$ and $p=p(n,2d)$.
\end{cor}

\begin{proof}
Write $\Sigma:=\Sigma_{n,2d}$, and let $t\le r\le p$. Then
$\Sigma(r-1)\ne\Sigma(r)$,
and hence $\interior \Sigma(r)\not\subset\Sigma(r-1)$ since
$\Sigma(r)=\ol{\interior \Sigma(r)}$ (Proposition \ref{equidim}). So
$\interior(\Sigma(r))\setminus\Sigma(r-1)$ is a non-empty open set of
forms of length equal to~$r$.
\end{proof}


\section{Lower bounds}\label{sectlowerbd}%

\begin{lab}\label{gramrappl}%
Let $k$ be a field of characteristic zero, let $V$ be a vector space
over $k$. For $m\ge0$ let $\sfS^mV$ be the $m$-th symmetric power of
$V$ over~$k$. We will identify $\sfS^mV$ with the subspace of
symmetric tensors in the $m$-fold tensor power $V^{\otimes m}$, which
is possible since $\ch(k)=0$.

Let $A$ be a $k$-algebra. The \emph{Gram tensor} of a given sum of
squares representation $f=\sum_{i=1}^ma_i^2$ in $A$ (with
$a_1,\dots,a_m\in A$) is the symmetric tensor
$$\theta\>=\>\sum_{i=1}^ma_i\otimes a_i\>\in\>\sfS^2A.$$
Of course we may as well regard $\theta$ as an element of $\sfS^2U$,
for any linear subspace $U\subset A$ containing $a_1,\dots,a_m$.
Two sum of squares representations $f=\sum_{i=1}^ma_i^2=\sum_{i=1}^m
b_i^2$ (with $a_i,\,b_i\in A$) are \emph{(orthogonally) equivalent}
if there exists an orthogonal matrix $u=(u_{ij})_{1\le i,j\le m}$
over $k$ (satisfying $uu^t=I$) such that $b_j=\sum_{i=1}^mu_{ij}a_i$
for $1\le j\le m$. Clearly, equivalent sum of squares representations
have the same Gram tensor. An elementary but important fact is that
the converse holds provided the field $k$ is real (see \cite{clr}).
\end{lab}

\begin{lab}
If $I\subset k[\x]=k[x_1,\dots,x_n]$ is a homogeneous ideal, we write
$h_j(I)=\dim k[\x]_j/I_j$ for $j\ge0$. For any set $Z\subset
\P^{n-1}(k)$, the full (saturated) vanishing ideal of $Z$ is denoted
$I(Z)$.
\end{lab}

We first discuss ternary forms (case $n=3$, $\x=(x_1,x_2,x_3)$), and
abbreviate $\Sigma_d:=\Sigma_{3,d}$.

\begin{prop}\label{langesos}%
Let $d\ge1$, and let $Z\subset\P^2(\R)$ be a set of $|Z|=
\choose{d+1}2$ real points in sufficiently general position. Then any
$f\in\Sigma_{2d}$ vanishing on $Z$ has a unique sum of squares
representation, up to orthogonal equivalence.
\end{prop}

\begin{proof}
Let $I=I(Z)$. It is enough to show that the product map $\mu\colon
\sfS^2(I_d)\to\R[\x]_{2d}$ is injective. Indeed, if $f\in\Sigma_{2d}$
vanishes on $Z$ and $f=\sum_\nu p_\nu^2$ is any sum of squares
representation, then $p_\nu\in I_d$ for all~$\nu$. By the asserted
injectivity of $\mu$, the symmetric Gram tensor $\sum_\nu
p_\nu\otimes p_\nu$ is uniquely determined by $f$, which is the claim
(see \ref{gramrappl}).

Since the points in $Z$ are general we have $I_{d-1}=\{0\}$ and
$\dim(I_d)=\choose{d+2}2-\choose{d+1}2=d+1$.
In particular, the subspace $\im(\mu)=I_dI_d$ of $\R[\x]_{2d}$
satisfies $I_dI_d=(I^2)_{2d}$.
Injectivity of $\mu$ means that this space has dimension equal to
$\dim\sfS^2(I_d)=\choose{d+2}2$, and hence is equivalent to
\begin{equation}\label{h2di2}%
h_{2d}(I^2)\>=\>\choose{2d+2}2-\choose{d+2}2\>=\>3\choose{d+1}2.
\end{equation}
Equality \eqref{h2di2} is known to be true, see Iarrobino-Kanev
\cite{ik} Prop.\ 4.8. In fact, the complete Hilbert series of $I^2$
is determined there. This proves Proposition \ref{langesos}.
\end{proof}

Proposition \ref{langesos} has the following consequences:

\begin{cor}
If $Z\subset\P^2(\R)$ is a set of $s$ general points where
$\choose{d+1}2\le s\le\choose{d+2}2$, then any general sum of
squares $f\in\Sigma_{2d}$ that vanishes on $Z$ has sos length
$\ell(f)=\choose{d+2}2-s$.
\end{cor}

\begin{proof}
Since $\dim I(Z)_d=\choose{d+2}2-s=:m$, the general $f\in
\Sigma_{2d}$ vanishing on $Z$ has the form $f=p_1^2+\cdots+p_m^2$
where $p_1,\dots,p_m$ form a basis of $I(Z)_d$.
\end{proof}

In particular we get a lower bound for the Pythagoras number:

\begin{cor}\label{lowerboundn3}%
Let $d\ge2$. Any general sum of squares of degree $2d$ in
$\R[x_1,x_2,x_3]$ vanishing in $\choose{d+1}2$ general $\R$-points
has sos length $d+1$. Therefore $p(3,2d)\ge d+1$.
\qed
\end{cor}

Combined with Corollary \ref{ternarybound} this gives:

\begin{thm}\label{p32d}
For any $d\ge2$, the Pythagoras number $p(3,2d)$ is either $d+1$ or
$d+2$.
\qed
\end{thm}

\begin{lab}
In principle the same argument extends to the case of $n\ge4$
variables. In general however, when $I$ is the vanishing ideal of a
generic finite set of points in $\P^{n-1}$, the Hilbert function of
$I^2$ is known only conjecturally. In \cite{ik} Sect.~3.2,
Iarrobino-Kanev formulate conjectures about the dimensions of tangent
spaces to catalecticant varieties corresponding to power sums. These
conjectures can be formulated in terms of the Hilbert function of
$I^2$. In particular, if $d$ is the lowest degree with $I_d\ne\{0\}$,
and if $b=\dim(I_d)$, it is conjectured that $(I^2)_{2d}=I_dI_d$ has
the maximum possible dimension $\choose{b+1}2$, unless a smaller
value is forced by the Alexander-Hirschowitz theorem. Explicitly,
this means (we write $N_d:=N_{n,d}=\choose{n+d-1}{n-1}$ in the
sequel):
\end{lab}

\begin{conj}\label{iakaconj}%
(Iarrobino-Kanev, \cite{ik} Conjecture 3.25)
Let $n\ge3$, $d\ge2$ and $N_{d-1}\le s<N_d$, and let $I=I(Z)$ be the
vanishing ideal of a set $Z$ of $s$ general points in $\P^{n-1}$.
Then
$$h_{2d}(I^2)\>=\>\max\Bigl\{ns,\>N_{2d}-\choose{N_d-s+1}2\Bigr\},$$
except for $(n,d,s)=(3,2,5)$, $(4,2,9)$ and $(5,2,14)$, where $\max$
has to be replaced by $\min$.
\end{conj}

Of course, the exceptional cases correspond to exceptional cases of
the Alexander-Hirschowitz theorem. Conjecture \ref{iakaconj} is
proved in \cite{ik} for $n=3$, and also for $n\ge4$ in the case
$N_d-n\le s<N_d$.
Note that \cite{ik} Conjecture 3.25 contains misprints, but compare
with loc.\,cit., Conjecture 3.20.

Assuming Conjecture \ref{iakaconj} we can conclude:

\begin{cor}\label{conjpythbound}%
Let $n\ge4$ and $d\ge2$, and let $s=s_{\min}(n,d)$ be the smallest
integer satisfying $\choose{N_d-s+1}2\le N_{2d}-ns$. Then $N_{d-1}<s
<N_d$. If Conjecture \ref{iakaconj} is true for the given values of
$(n,d,s)$, then
$$p(n,2d)\>\ge\>N_d-s.$$
In fact, then, any generic sum of squares in $\Sigma_{2d}$ that
vanishes in a generic set of $s$ real points has sos length equal to
$N_d-s$.
\end{cor}

\begin{proof}
We only sketch the proof of $N_{d-1}<s<N_d$. Check the claim directly
for $(n,d)=(4,2)$ or $(5,2)$ and discard these cases. Consider the
polynomial
$$P(x)\>=\>x^2-(2N_d-2n+1)x+N_d(N_d+1)-2N_{2d}.$$
By definition, $s=s_{\min}(n,d)$ is the smallest integer satisfying
$P(s)\le0$. One shows that $P(N_d-1)\le0$, thereby proving $s<N_d$.
To prove the other inequality, fix $n$ and consider $P(N_{d-1})$ as a
polynomial in $d$. One shows that $P(N_{d-1})=(d-1)Q(d)$ where $Q$ is
a polynomial with all coefficients nonnegative.
This implies $P(N_{d-1})>0$, and thus $s>N_{d-1}$. The details are
left to the reader.

By the argument in the proof of \ref{langesos} we get $p(n,2d)\ge
N_d-t$ for any number $t$ such that the map $\sfS^2I_d\to\R[\x]_{2d}$
is injective, where $I$ is the vanishing ideal of a generic set of
$t$ points in $\P^{n-1}(\R)$.
Since $s=s_{\min}(n,d)\ge N_{d-1}$, Conjecture \ref{iakaconj}
predicts that this property holds for $t=s$.
\end{proof}

\begin{rems}
\hfil
\smallskip

1.\
In Conjecture \ref{iakaconj} it is clear that $\ge$ holds, i.e.\
that $h_{2d}(I^2)$ is at least the right hand maximum (resp.\ minimum
in the exceptional cases).
For any concrete values of $n$, $d$ and $s$, the conjecture can be
verified by finding a single concrete set $Z$ with $|Z|=s$ for which
equality holds in \ref{iakaconj}. In this way the conjecture is
easily verified for sufficiently small values of $n,\,d,\,s$, using a
computer algebra system.
\smallskip

2.\
An unconditional formulation of Corollary \ref{conjpythbound}, not
depending on Conjecture \ref{iakaconj}, would be:
\emph{Let $n\ge4$ and $d\ge2$, let $N_{d-1}\le s<N_d$, and let $I$ be
the vanishing ideal of a generic set $Z\subset\P^{n-1}(\R)$ of $s$
points. If $h_{2d}(I^2)=N_{2d}-\choose{N_d-s+1}2$, then $p(n,2d)\ge
N_d-s$.}
The drawback, of course, is that we have no good control of what
numbers $s$ satisfy the condition. Theorem 4.19 in \cite{ik} shows
that $s=N_d-n$ is admitted, but this only gives the useless bound
$p(n,2d)\ge n$.
\end{rems}

\begin{rem}
For small values of $n\ge4$ and $d\ge2$ we record the bounds on
$p(n,2d)$ that we have obtained. The following table lists the
minimal number $s=s_{\min}(n,d)$ and the corresponding lower bound
$N_d-s$ for $p(n,2d)$ from Corollary \ref{conjpythbound}. We compare
these with the upper bounds from Section \ref{secreview}. (Those
upper bounds are usually the numbers $\lfloor\Lambda(n,2d)\rfloor$,
with only three exceptions where $L(n,2d)$ is better, see Remark
\ref{uppbdvgl}.)
$$\begin{array}{l|ccccccc}
\qquad d\qquad & 2 & 3 & 4 & 5 & 6 & 7 & 8 \\[2pt]
\hline
s_{\min}(4,d) & 5 & 12 & 24 & 41 & 65 & 97 & 137 \\
p(4,2d)\ge{} & 5 & 8 & 11 & 15 & 19 & 23 & 28 \\
p(4,2d)\le{} & 7 & 11 & 16 & 22 & 29 & 36 & 43 \\[3pt]
s_{\min}(5,d) & 8 & 21 & 48 & 94 & 166 & 273 & 422 \\
p(5,2d)\ge{} & 7 & 14 & 22 & 32 & 44 & 57 & 73 \\
p(5,2d)\le{} & 11 & 20 & 30 & 44 & 59 & 77 & 97 \\[3pt]
s_{\min}(6,d) & 10 & 34 & 88 & 192 & 374 & 670 & 1123 \\
p(6,2d)\ge{} & 11 & 22 & 38 & 60 & 88 & 122 & 164 \\
p(6,2d)\le{} & 15 & 29 & 50 & 77 & 110 & 152 & 201
\end{array}$$
\end{rem}

\begin{rem}
The forms in $\Sigma_{n,2d}$ of large sos length that were
constructed in Corollaries \ref{lowerboundn3} and
\ref{conjpythbound} are of very special nature, in that they have
many real zeros. Corollary \ref{typle} on typical lengths shows
that there exists a nonempty open set of nonsingular forms with the
same sos length.
\end{rem}

\begin{thm}\label{asymptbds}%
Let $n\ge4$. Assuming Conjecture \ref{iakaconj}, the Pythagoras
number $p(n,2d)$ grows asymptotically like $d^{(n-1)/2}$ for
$d\to\infty$. More precisely, for any $\epsilon>0$ the inequalities
$$(c_n-\epsilon)\,d^{(n-1)/2}\><\>p(n,2d)\><\>(C_n+\epsilon)\,
d^{(n-1)/2}$$
hold for all but finitely many $d$, where
$$c_n\>=\>\sqrt{\frac{2^n-2n}{(n-1)!}},\quad C_n\>=\>
\sqrt{\frac{2^n}{(n-1)!}}.$$
\end{thm}

\begin{proof}
Let $\theta(n,2d):=N_d-s_{\min}(n,d)$, see \ref{conjpythbound}. For
all $d\ge2$ we have $\theta(n,2d)\le p(n,2d)\le\Lambda(n,2d)$. Since
$N_{2d}=\frac1{(n-1)!}(2d+1)\cdots(2d+n-1)$, it is clear that
$$\Lambda(n,2d)\ \sim\ \sqrt{2N_{2d}}\ \sim\ \sqrt{\frac2{(n-1)!}
(2d)^{n-1}}\>=\>C_nd^{(n-1)/2}$$
for $d\to\infty$ (meaning that the quotient of both sides converges
to~$1$).
On the other hand, $\theta=\theta(n,2d)$ is the largest
integer $0<\theta<N_d$ satisfying $\choose{\theta+1}2\le N_{2d}
-n(N_d-\theta)$. This means
$$\theta\>\approx\>n-\frac12+\sqrt{2N_{2d}-2nN_d+n^2-n+\frac14},$$
and since
$$N_{2d}-nN_d\>=\>\frac1{(n-1)!}\Bigl((2d+1)\cdots(2d+n-1)-n(d+1)
\cdots(d+n-1)\Bigr)$$
we get $\theta(n,2d)\sim c_nd^{(n-1)/2}$.
\end{proof}

\begin{rem}
A different view on the asymptotic behaviour of $p(n,2d)$ is taken in
\cite{clr} Theorem 6.4.
Fixing the degree, it is shown there that $\gamma_1\,n^d\le p(n,2d)
\le\gamma_2\,n^d$ for all $n\ge1$ and suitable $\gamma_1$,
$\gamma_2>0$ depending on $d$. This result only uses the weak lower
bound $\lambda(n,2d)$.
\end{rem}


\section{Ternary sextics and quaternary quartics}\label{sec3644}%

We now prove that the lower bounds for the Pythagoras number from
Section~\ref{sectlowerbd} are sharp, in the case of ternary sextics
or quaternary quartics:

\begin{thm}\label{p36}%
$p(3,6)=4$: Every ternary sextic that is a sum of squares is a sum of
four squares.
\end{thm}

\begin{thm}\label{p44}%
$p(4,4)=5$: Every quaternary quartic that is a sum of squares is a
sum of five squares.
\end{thm}

\begin{cor}
The elements of length~$4$ in $\Sigma_{3,6}$, or those of length~$5$
in $\Sigma_{4,4}$, form an open dense subset.
\end{cor}

\begin{proof}
The set of three squares in $\Sigma_{3,6}$ (resp.\ of four squares in
$\Sigma_{4,4}$) is nowhere dense since $\lambda(3,6)=4$ (resp.\
$\lambda(4,4)=5$). So the claims follow from \ref{p36} resp.\
\ref{p44}.
\end{proof}

The proofs of Theorems \ref{p36} and \ref{p44} are similar. We do the
$(3,6)$ case first, and then explain how to modify the argument in the
$(4,4)$ case. See Remark \ref{pfrem} below for comments on the proofs.

\begin{lab}\label{gorideal}%
In the following let $k$ be a field, $\ch(k)=0$. Write $A=k[\x]=
k[x_1,\dots,x_n]=\bigoplus_{e\ge0}A_e$. Given $e\ge1$ and $0\ne\alpha
\in A_e^\du=\Hom(A_e,k)$, let $I(\alpha)\subset A$ be the graded ideal
defined by $I(\alpha)_i=\{p\in A_i\colon pA_{e-i}\subset
\ker(\alpha)\}$ ($0\le i\le e$) and $I(\alpha)_i=A_i$ for $i>e$. Then
the graded artinian ring $R=A/I(\alpha)$ is Gorenstein with socle
degree~$e$. In particular, this means that $\alpha$ induces a linear
duality between $R_i$ and $R_{e-i}$ for $0\le i\le e$.

The following is the key observation in the $(3,6)$ case:
\end{lab}

\begin{prop}\label{keylem36}%
Let $p_1,\dots,p_r\in k[x_1,x_2,x_3]$ be linearly independent cubic
forms such that $\langle x_1,x_2,x_3\rangle^6\not\subset\langle
p_1,\dots,p_r\rangle$. If $r\ne3$ then $p_1,\dots,p_r$ have a common
zero (in~$\ol k$). In particular, then, the form $p_1^2+\cdots+p_r^2$
is singular.
\end{prop}

\begin{proof}
Write $\x=(x_1,x_2,x_3)$, let $U\subset k[\x]_3$ be the linear span
of $p_1,\dots,p_r$. By assumption there exists a linear functional
$0\ne\alpha\in(k[\x]_6)^\du$
with $Uk[\x]_3\subset\ker(\alpha)$. Let $I=I(\alpha)\subset k[\x]$ be
the Gorenstein ideal defined by $\alpha$. Then $U\subset I_3$. It is
enough to assume that $p_1,\dots,p_r$ have no common zero and to show
$r\le3$.
By the assumption there exist three forms $q_1,\,q_2,\,q_3\in U$
without common zero. Hence the ideal $J:=\langle q_1,q_2,q_3\rangle$
is a complete intersection, so $J$ is an artinian Gorenstein ideal
with socle degree~$6$ (see \cite{egh} Theorem CB8).
By construction we have $J\subset I$. On the other hand, both $I$ and
$J$ are Gorenstein ideals with same socle degree. Therefore $J=I$,
and in particular, $I_3=J_3$ has dimension~$3$, whence $r\le3$.
\end{proof}

\begin{lab}\label{36proof}%
We now give the proof of Theorem \ref{p36}. Let $\Sigma=
\Sigma_{3,6}$, and let $\Sigma(4)\subset\Sigma$ be the set of sums of
four squares. The map
$$\phi\colon(\R[\x]_3)^4\to\R[\x]_6,\quad\phi(p_1,p_2,p_3,p_4)=
\sum_{i=1}^4p_i^2$$
is proper, and its image set $\Sigma(4)$ has nonempty interior in
$\R[\x]_6$. Let $Z\subset\Sigma(4)$ be the set of critical values of
$\phi$. So $Z$ consists of all forms $\phi(p)$ where
$p=(p_1,p_2,p_3,p_4)\in(\R[\x]_3)^4$ is such that $\langle p_1,p_2,
p_3,p_4\rangle_6\ne\R[\x]_6$.
The set $Z$ is closed and semi-algebraic in
$\R[\x]_6$ and has empty interior, e.g.\ by Sard's theorem.
Whenever $[0,1]\to\R[\x]_6$, $t\mapsto f_t$ is a smooth path in
$\R[\x]_6$ with $f_t\in\interior(\Sigma)$ and $f_t\notin Z$ for all
$0\le t\le1$, then $f_0\in\Sigma(4)$ implies $f_1\in\Sigma(4)$, since
one can lift the path $t\mapsto f_t$ to a path in $(\R[\x]_3)^4$.

Since $\lambda(3,6)=4$, it is clear that $p(3,6)\ge4$. We assume
$p(3,6)\ge5$ and will arrive at a contradiction. By assumption the
semi-algebraic set $Z':=\interior(\Sigma)\cap\Sigma(4)\cap
\ol{(\Sigma\setminus\Sigma(4))}$ is non-empty. It is contained in
$Z$ and has codimension one in $\R[\x]_6$, by the above path argument
and since the complement of a codimension two set in
$\interior(\Sigma)$ is connected.
So there exists $f\in Z'$, together with an open neighborhood $W$ of
$f$ contained in $\Sigma$, such that $Z'\cap W$ is a smooth
codimension one submanifold of $W$. Since every singular form in
$\interior(\Sigma)$ has at least two distinct singularities (complex
conjugate), the set $\{f\in\interior(\Sigma)\colon f$~is singular$\}$
has codimension two. Therefore we can choose $f$ to be nonsingular.
Then it follows from Proposition \ref{keylem36} that $f$ is a sum of
three squares.

Let $0\ne\alpha\in(\R[\x]_6)^\du$ such that $\{g\in\R[\x]_6\colon
\alpha(g-f)=0\}$ is the affine hyperplane in $\R[\x]_6$ tangent to
$Z'$ at $f$. If we shrink $W$ appropriately, the hypersurface $Z'$
divides $W$ into two open halves, one contained in $\Sigma(4)$, the
other contained in $\Sigma\setminus\Sigma(4)$. Replacing $\alpha$
with $-\alpha$ if necessary, we can thus assume: If $h\in\R[\x]_6$
satisfies $\alpha(h)>0$, then $\ell(f+th)\le4$ and $\ell(f-th)\ge5$,
for all sufficiently small $t>0$.
It follows that $\alpha(f)=0$.
Since $f\in\interior(\Sigma)$, there exists $p\in\R[\x]_3$ with
$\alpha(p^2)<0$. By the choice of $\alpha$ we have $\ell(f+tp^2)\ge5$
for small $t>0$. But this is a contradiction since $\ell(f)\le3$.
Theorem \ref{p36} is proved.
\end{lab}

In the $(4,4)$ case, the analogue to Proposition \ref{keylem36} is

\begin{prop}\label{keylem44}%
Let $p_1,\dots,p_r\in k[x_1,x_2,x_3,x_4]$ be linearly independent
quadratic forms such that $\langle x_1,x_2,x_3,x_4\rangle^4
\not\subset\langle p_1,\dots,p_r\rangle$.
If $r\ne4$ then $p_1,\dots,p_r$ have a common zero. In
particular, then, the form $p_1^2+\cdots+p_r^2$ is singular.
\end{prop}

The proof is completely parallel to the proof of \ref{keylem36}. The
essential point is that four quadratic forms in $k[\x]=k[x_1,\dots,
x_4]$ without common zero
define a complete intersection with socle degree~$4$. Having proved
\ref{keylem44}, one considers the sum of squares map $\phi\colon
(\R[\x]_2)^5\to\R[\x]_4$ and proceeds similarly to \ref{36proof},
thereby proving Theorem \ref{p44}.

\begin{rem}\label{pfrem}%
The proofs of Theorems \ref{p36} and \ref{p44} have much in common
with Hilbert's argument for showing that every nonnegative ternary
quartic is a sum of three squares \cite{hi}. Hilbert considered the
sum of squares map $\phi\colon(\R[\x]_2)^3\to\R[\x]_4$
($\x=(x_1,x_2,x_3)$) and showed that the critical values of $\phi$
are singular forms. Hence the strictly positive ones among them form
a codimension two subset of $\R[\x]_4$. This implies that the set of
strictly positive forms that are not critical values of $\phi$ is
connected, leading to the desired conclusion.

To prove \ref{p36}, say, we have shown that the critical values of
$(\R[\x]_3)^4\to\R[\x]_6$ are singular or else sums of three squares.
We then needed an extra argument to deal with the second case.
\end{rem}

\begin{rem}
Blekherman proved that every positive definite form in the boundary
$\partial(\Sigma_{3,6})$ (resp.\ in $\partial(\Sigma_{4,4})$) is a
sum of three (resp.\ four) squares (\cite{bl}, Corollaries 1.3 and
1.4). Propositions \ref{keylem36} and \ref{keylem44} give a quick way
of reproving these results.
We explain this for $(n,d)=(4,4)$. Let $B$ be the set of positive
definite forms in $\partial(\Sigma_{4,4})$. Every singular form in
$B$ has at least two different singularities (complex conjugate), so
these forms constitute a subset of codimension $\ge2$ in $\R[\x]_4$.
Since the semi-algebraic set $B$ has codimension one in $\R[\x]_4$
locally at each of its points, every form in $B$ is a limit of
nonsingular forms in $B$. Every nonsingular form in $B$ is a sum of
four squares by \ref{keylem44}. Since the sums of four squares form a
closed set, it follows that $B$ consists of sums of four squares.
\end{rem}

\begin{rem}
Unfortunately, Propositions \ref{keylem36} and \ref{keylem44} do not
directly extend to higher degrees. Still we conjecture in the ternary
case that the lower bound \ref{lowerboundn3} is sharp, i.e.\ that
$p(3,2d)=d+1$ holds for all $d\ge2$.
\end{rem}


\end{document}